\documentclass[11pt]{amsart} 

\usepackage[margin=2.9cm]{geometry}             

\usepackage{amsmath}
\usepackage{amssymb}
\usepackage{amsthm}
\usepackage{theoremref}
\usepackage{enumerate}

\theoremstyle{definition}
\newtheorem{theo}{Theorem}[section]
\newtheorem{lem}[theo]{Lemma}

\newtheorem{prop}[theo]{Proposition}

\newtheorem{rem}[theo]{Remark}

\numberwithin{equation}{section}

\theoremstyle{definition}
\newtheorem{defi}[theo]{Definition}

\newtheorem{ex}[theo]{Example}

\newcommand{\R}{\mathbb{R}}
\newcommand{\Z}{\mathbb{Z}}
\newcommand{\C}{\mathbb C} 
\newcommand{\N}{\mathbb N} 
 
\newcommand{\D}{\mathbb D}
\newcommand{\HH}{\mathbb{H}}

\usepackage{graphicx}
\usepackage{float}
\usepackage{subcaption}
\graphicspath{ {./Roots of Unity/} }
\usepackage{listings}
\usepackage{xcolor}

\begin{document}

\author[Ban, Bertrand, Jaber Chehayeb, Salha, Tabbara]{Seok Ban, Florian Bertrand, Amir Jaber Chehayeb, Adam Salha, Walid Tabbara}
\title{On the higher order Kobayashi pseudometric}

\begin{abstract}
We study the higher order Kobayashi pseudometric introduced by Yu. We first obtain estimates of this pseudometric in a special pseudoconvex domain in $\C^3$. We then study the structure of the higher order extremal discs and their connection with the standard extremal discs for the  Kobayashi metric. 
\end{abstract} 
\thanks{Research of the second  author was  supported by the Center for Advanced Mathematical Sciences}

\maketitle 

\section*{Introduction}
In his thesis and related papers \cite{yu,yu1,yu2} Yu introduced a higher order version of the Kobayashi pseudometric  for the purpose of measuring precisely the type invariants of  boundaries of pseudoconvex domains. This invariant pseudometric shares many of the fundamental properties of the standard Kobayashi pseudometric; in particular, in the complex plane, 
the higher order Kobayashi pseudometric and the Kobayashi pseudometric coincide on any domain (see for instance \cite{ja-pf}).  The pseudometric introduced by Yu was later on studied by many different authors \cite{ki-hw-ki-le,ki,ni1,ni2,ja-pf}. In this paper, we focus on two aspects of this pseudometric. We first consider a particular domain given by Yu \cite{yu} and obtain new estimates of the higher order Kobayashi metric (Theorem \ref{theoest}); as an application of our method, we obtain the exact value of the Kobayashi metric in certain cases (Proposition \ref{propexact}). This answers partially  two questions addressed by Jarnicki and Pflug (problems 3.7 and 3.8 p.149 \cite{ja-pf}). Then, inspired by the works of Poletski \cite{po}, Edigarian \cite{ed}, Jarnicki and Pflug \cite{ja-pf}, and the paper \cite{be-de-jo}, we study the higher order pseudometric by means of the corresponding extremal discs and focus on their connection with the usual extremal discs (Theorem \ref{theopext}).  We conclude the paper with an appendix in which we prove higher order Schwarz type lemmas. 

\section{Preliminaries}

We denote by $\D=\{\zeta \in \C \ | \ |\zeta|<1\}$ the unit disc in $\C$ and  by $\D_r=\{\zeta \in \C \ | \ |\zeta|<r\}$ the disc centered at $0$ and radius $r>0$.   

\subsection{Invariant pseudometrics and the Kobayashi pseudometric of higher order}

Let $\Omega \subset \C^n$ be a domain. Following \cite{ko,ro}, the {\it Kobayashi pseudometric} of the domain $\Omega$ at $p\in \Omega$ and 
$v \in T_p\Omega=\C^n$ is defined by
$$K_{\Omega}\left(p,v\right)=\inf 
\left\{\frac{1}{r}>0 \ \big| \  \exists \ f: \D \to \Omega \ \mbox{holomorphic}, f\left(0\right)=p, f'(0)=rv\right\}.$$
We also recall  the {\it Carath\'eodory pseudometric} $C_\Omega$ at $p \in \Omega$ and $v \in \C^n$ defined by
$$C_\Omega(p,v) = \text{sup}\{\left|d_pg (v)\right| \ | \ g: \Omega \rightarrow \D \ \mbox{holomorphic}, g(p) = 0\},$$
where $d_pg (v)$ is the differential at $p$  of $g$ in the direction $v$.  
Note that in the case of the unit disc $\D \subset \C$, it follows from Schwarz Lemma that both  $K_\D$  and $C_\D$  coincide with the Poincar\'e metric, that is, for $\zeta \in \D$ and $v \in \C$ 
$$K_\D(\zeta,v) = C_{\D}(p,v)=\cfrac{|v|}{1 - |\zeta|^2}.$$

Kobayashi pseudometrics of higher order were introduced by different authors \cite{ve,yu}. In this paper, we focus on the one defined by Yu in \cite{yu,yu1,yu2}. More precisely, for a positive integer $k>0$, the {\it $k^{th}$ order Kobayashi pseudometric} is defined by 
$$K^k_{\Omega}\left(p,v\right)=\inf
\left\{\frac{1}{r}>0 \ | \  \exists \ f: \D \to \Omega \ \mbox{holomorphic}, f\left(0\right)=p, \nu(f)\geq k, f^{(k)}(0)=k!rv\right\},$$
where $p\in \Omega$, $v \in \C^n$, and $\nu(f)$ is the vanishing order of $f$ defined as the degree of the first nonzero  term in the power expansion of $f$. 

Before stating some fundamental properties of the higher order Kobayashi metric, we first establish the following higher order Schwarz Lemma on the unit disc.
\begin{lem}[Higher order Schwarz Lemma]\label{lemhigh} 
Let $f:\D \rightarrow \D$ be a holomorphic function satisfying $f^{(\ell)}(0) = 0$  for all $\ell=0,\ldots,k-1$. Then we have for all $\zeta \in \D$
$$|f(\zeta)| \leq |\zeta|^{k}$$
and 
 $$|f^{k}(0)|\leq k!$$
Moreover, if $|f(\zeta)| = |\zeta|^k$ for some $\zeta \neq 0$ or $|f^{(k)}(0)| = k!$, then $f(\zeta) = e^{i\theta}\zeta^k$ for some $\theta \in \R$. 
\end{lem} 
Note that the proof of this lemma is essentially contained in the proof of Proposition 2.2 in \cite{yu}. We give a slightly different proof. 
\begin{proof}
We proceed by induction. For $k = 1$, this is the classical Schwarz Lemma. Assume that the $k^{th}$ order Schwarz Lemma holds. Let $f:\D \rightarrow \D $ be a holomorphic function with $f^{(\ell)}(0) = 0$ for  $\ell=0, \dots k$. Define the disc $g(\zeta)=f(\zeta)/\zeta^k$. According to the $k^{th}$-order Schwarz Lemma, 
we have $g: \D \to \D$. Since $g(0)=0$, the classical Schwarz Lemma applied to $g$ leads to $|f(\zeta)| \leq |\zeta|^{k+1}$ for all $\zeta \in \D$ and 
$$|g'(0)|=\frac{|f^{k+1}(0)|}{(k+1)!}\leq 1.$$
In case of equality, the result follows once again by induction.  
\end{proof}

We now summarize important basic properties of the higher order Kobayashi pseudometric which have been established by many authors \cite{yu,ja-pf}.
\begin{prop}[\cite{yu,ja-pf}]\label{propprop}
Let $\Omega,\Omega' \subset \C^n$ be two domains. Let $k>0$ be a positive integer.  
\begin{enumerate}[i.]
\item Let $F:\Omega \to \Omega'$ be a holomorphic map. Then for any  $p\in \Omega$ and $v \in T_p\Omega$, we have
$$K^k_{\Omega'}(F(p),d_pF(v))\leq K^k_{\Omega}(p,v).$$
In particular,  the higher order Kobayashi pseudometric is invariant under biholomorphisms.
\item In case $n=1$, the pseudometrics $K^k_{\Omega}$ and $K_{\Omega}$ coincide.
\item We have $K^{mk}_\Omega \leq K^{k}_\Omega$ for any $m \in \; \N$.
\item For any  $p\in \Omega$ and $v \in T_p\Omega$, we have 
$$C_{\Omega}(p,v) \leq K^k_{\Omega}(p,v) \leq K_{\Omega}(p,v).$$ 
\end{enumerate}
\end{prop}
For completeness, we include the proof.   
\begin{proof} {\phantom{m}}
We first prove {\rm i}. 
 Let $g: \D \rightarrow \Omega$ be a holomorphic disc satisfying $g(0) = p$ and 
 $g^{(k)}(0) = k!r v$ for some $r>0$. Consider the composition $F\circ g: \D \rightarrow \Omega'$ and note that $(F\circ g)(0) = F(p)$ and 
 $$(F\circ g)^{(k)}(0) = d_pF(g^{(k)}(0)) = k!r  d_pF(v).$$ It follows directly that $K_{\Omega'}^k(F(p),d_pF(v)) \leq K_\Omega^k(p,v)$. 
   
The proof of {\rm ii} follows from Proposition 3.8.8 in \cite{ja-pf} and the uniformization theorem. The fact that $K^k_{\D}$ and $K_{\D}$ coincide follows directly from the higher order Schwarz Lemma \ref{lemhigh} since by invariance of both metric, it is enough to show that 
$K^k_{\D}(0,v)=K_{\D}(0,v)=|v|$. Their equality on the punctured disc follows, for instance, from Lemma \ref{lempdhigh}. 

For   {\rm iii}, we follow exactly \cite{ja-pf}. Let $p \in \Omega$ and $v \in T_p\Omega$ and let $f:\D \to \Omega$ be a holomorphic disc with $f\left(0\right)=p$, $\nu(f)\geq k$, and $f^{(k)}(0)=k!rv$. We define $g(\zeta)=f(\zeta^m)$.  The holomorphic disc $g:\D\to \D$ satifsies 
$g\left(0\right)=0$, 
$\nu(g)\geq km$, and 
$$g^{(km)}(0)=(mk)!\frac{f^{(k)}(0)}{k!}=(mk)!rv.$$ Thus $K^{mk}_\Omega (p,v) \leq K^{k}_\Omega(p,v)$.  
 
Finally, we prove  {\rm iv}. The inequality $K^k_{\Omega}(p,v) \leq K_{\Omega}(p,v)$ follows immediately from  {\rm iii}. 
For the inequality concerning the Carath\'eodory metric, we consider $f: \D \to \Omega$ and $ g:\Omega \to \D$ both holomorphic and satisfying 
$f\left(0\right)=p, \nu(f)\geq k, f^{(k)}(0)=k!rv, \mbox{and } g(p)=0$. The composition $g\circ f:\D\to\D$ is such that $g\circ f(0)=0$ and, since $\nu(f)\geq k$, we have 
$(g\circ f)^{(l)}(0)=0$ for $l<k$. We then apply the higher order Schwarz Lemma \ref{lemhigh} to obtain $|(g\circ f)^{(k)}(0)| \leq k!$, and since 
$$(g\circ f)^{(k)}(0)= d_{f(0)}g(f^{(k)}(0))= d_{p}g(k!rv),$$
we get $|d_{p}g(v)|\leq \frac{1}{r}.$
This shows  {\rm iv} and ends the proof
    
\end{proof}

\begin{rem} The following higher order Kobayashi metric was introduced by Venturini in \cite{ve}
$$F^k_{\Omega}\left(p,\xi\right):=\inf 
\left\{\frac{1}{r}>0 \ | \  f: \D \rightarrow \Omega
\mbox{ holomorphic}, f\left(0\right)=p, f^{\ell}(0)=r^\ell\xi_\ell, 1 \leq \ell \leq k\right\}$$
for any $p\in M$ and any $k$-jet  $\xi=(\xi_1,\ldots,\xi_k)\in J^k_p(\Omega)=\C^{kn}$ at $p$. 
The relation between these two higher order Kobayashi pseudometrics is given by 
$$k!\left(F^k_{\Omega}\left(p,(0,\cdots,0,v)\right)\right)^k=K^k_{\Omega}\left(p,v\right).$$
\end{rem}

\subsection{Higher order extremal and stationary discs}
Following Lempert \cite{le}, we define higher order extremal discs as follows:
\begin{defi} Let $k>0$ be a positive integer and let $\Omega \subset \C^n$ be a domain.  A holomorphic disc $f: \D \rightarrow \Omega$ is a {\it  $k$-extremal disc for the pair $(p,v) \in \Omega \times T_p\Omega$} if  $f(0)=p$, $\nu(f)\geq k$, $f^{(k)}(0)=k!\lambda v$ 
with $\lambda>0$ and if
 $g: \D \rightarrow \Omega$ is  holomorphic and such that $g(0)=p$, $\nu(g)\geq k$, $g^{(k)}(0)=k!\mu v$ with $\mu>0$, then $\mu\leq \lambda$. In case $k=1$, we will simply say that $f$ is an {\it   extremal disc}.
 We  denote by $X_{\Omega}^k(p,v)$ the set of  $k$-extremal discs for the pair $(p,v)$, and we set $X_{\Omega}^1(p,v)=X_{\Omega}(p,v).$
\end{defi}

Note that in case the domain $\Omega$ is bounded, then by Montel's theorem, for any pair
$(p,v) \in \Omega \times T_p\Omega$   there exists a corresponding  $k$-extremal disc. 
The following result  is a direct consequence of  Propostion \ref{propprop} $i.$ 
\begin{lem}\label{leminv}
Let $\Omega \subset \C^n$ be a domain and let $F$ be an automorphism of $\Omega$. Then for any $p \in D, v \in T_p\Omega$ we have 
$$X_\Omega^k(F(p),d_pF(v)) = F_* X_\Omega^k(p,v).$$ 
\end{lem}

\vspace{0.5cm}
Let $\Omega = \{\rho<0\}\subset \mathbb C^n$ be a smooth domain, where $\rho$ is a defining function. We set 
$\displaystyle \partial \rho = \left(\frac{\partial \rho}{\partial z_1},\ldots,\frac{\partial \rho}{\partial z_n}\right)$. Let $k$ be 
a positive integer. Following \cite{be-de}, we define
\begin{defi}
A bounded holomorphic map $f:\Delta\to \mathbb C^n$ is a \emph{$k$-stationary disc attached to $b\Omega$ in the $L^{\infty}$ sense}  if $f(b\Delta)
\subset b\Omega$ a.e. and if there exists a  $L^{\infty}$ function $c:b\Delta\to \mathbb R^+$ such that the function $ \zeta\mapsto \zeta^k 
c(\zeta)\partial \rho(f(\zeta))\in \mathbb C^n$ defined on $b\Delta$ extends holomorphically to $\Delta$.
\end{defi}
These discs were introduced in \cite{be-de} and generalize the notion of stationary discs introduced by Lempert \cite{le}. They are  particularly well adapted to study Levi degenerate hypersurfaces.

\section{An example of Yu}

In this section, we are interested in the behavior of the higher order Kobayashi pseudometric in a particular domain. Following Yu \cite{yu}, we consider the  domain 
$\Omega =\{z=(z_1, z_2, z_3) \in \C^3 \ | \  \rho(z,\overline z)<0\} \subset \C^3$, where 
$$ \rho(z,\overline z) = \Re e z_3 + \big|z_1^2 - z_2^3\big|^2.$$
The defining function $\rho$ is plurisubharmonic, and thus, the domain $\Omega$ is pseudoconvex. This example is the first example of a domain in which the higher order Kobayashi metric does not coincide with the usual Kobayashi metric (see Proposition \ref{propyu} below).

\subsection{Estimates of the higher order Kobayashi pseudometric}

Our goal is to study the higher order Kobayashi pseudometric for the pair 
$$(p,v) = \left((0,0,-1), (0,1,0)\right) \in \Omega \times \C^3.$$ 
The following proposition was proved in \cite{yu}.
\begin{prop}[\cite{yu}]\label{propyu} We have 
$$K_\Omega\left((0,0,-1), (0,1,0)\right) = 1,$$
and for any positive integer $k>0$
$$K^{2k}_\Omega\left((0,0,-1), (0,1,0)\right) = 0.$$
\end{prop}

As pointed out in the book \cite{ja-pf}, it is relevant to find good estimates of $K^k_\Omega\left((0,0,-1), (0,1,0)\right)$ for $k\geq3$ odd. We will now focus on this question. A first important observation is: 
\begin{lem}
For any positive integer $k >0$, we have  
$$K^{2k+1}_\Omega\left((0,0,-1), (0,1,0)\right) \leq K^{2k- 1}_\Omega\left((0,0,-1), (0,1,0)\right).$$
\end{lem}

\begin{proof} 
Recall that $(p,v) = \left((0,0,-1), (0,1,0)\right)$. Let $ f = (f_1, f_2, f_3) : \D \to \Omega$ be a holomorphic disc satisfying 

\begin{equation*}
\begin{cases}
f(0) = p\\
f^{(\ell)}(0) = 0, \ \ell = 1,\ldots, 2n-2\\
f^{(2n-1)}(0) = (2n-1)!rv\\
\end{cases}
\end{equation*} 
for some $r>0$. Consider the holomorphic disc $g: \D \to \C^3$ defined by 
$$g(\zeta) = (\zeta^{3}f_1(\zeta), \zeta^2f_2(\zeta), f_3(\zeta))$$
for $\zeta \in \D.$ We first note that  $g: \D \to \Omega$ since 
\begin{eqnarray*}
 \rho\left(g(\zeta),\overline{g(\zeta)}\right) &= & \Re e f_3(\zeta) + \big |\zeta^6f_1^2(\zeta) - \zeta^6{f_2}^3(\zeta)\big|^2 \\
 & = & \Re e f_3(\zeta) + \big|\zeta\big|^{12}\big| {f_1}^2(\zeta) - {f_2}^3(\zeta)\big|^2 \\
 & < & \rho\left(f(\zeta),\overline{f(\zeta)}\right)\\
\end{eqnarray*}
which is negative since $f:\D \to \Omega$. Moreover we have $g(0) = p$, $ g^{(\ell)}(0) = 0$  for any $\ell =1,\ldots, 2n$, and $g^{(2n+1)}(0) = (2n+1)!rv.$
This proves the lemma. 
\end{proof}
Accordingly, we will then focus on estimating $K^{3}_\Omega\left((0,0,-1), (0,1,0)\right)$. We notice the basic estimate
\begin{equation}\label{eqbasic}
K^{3}\left((0,0,-1), (0,1,0)\right) \leq 1
\end{equation}
which simply follows from the previous lemma and Proposition \ref{propyu}.
Our main result is 
\begin{theo}\label{theoest}
The $3^{rd}$ order Kobayashi pseudometric satisfies  
\begin{equation}\label{eqK3}
K^3_\Omega\left((0,0,-1), (0,1,0)\right)\leq   \left(\cfrac{8\pi}{1-e^{-2\pi}}\right)^{-1/3} \approx 0.3412.
\end{equation}
\end{theo} 
\begin{proof}
The idea is to find relevant holomorphic discs contained in $\Omega$. Instead of simply giving the explicit expression of the disc leading to the estimate  (\ref{eqK3}), we explain its 
construction. 

We seek a holomorphic disc $f = (f_1, f_2, f_3): \D \to \Omega$ satisfying  
\begin{equation}\label{eqcond} 
\begin{cases}
f(0) = (0,0,-1) \\
f^{(\ell)}(0) = 0, \ \ell = 1, 2\\
\displaystyle f^{(3)}(0)= 6r (0,1,0)\\
\end{cases}
\end{equation}
with $r>0$. We will fix $f_3 \equiv -1$.  Such a disc can be written as 
$$f(\zeta)= \left(f_1(\zeta), f_2(\zeta), -1\right) = \left(\zeta ^4 h_1(\zeta), \zeta^3 h_2(\zeta), -1\right)$$ 
for some holomorphic functions $h_1, h_2: \D \to \C$.  We assume that $h_1(0) = 1$ and we 
write 
$$h_1(\zeta) = 1 + \zeta\varphi(\zeta)$$
for some holomorphic function $\varphi$. The problem is then to find two holomorphic functions $\varphi$ and $h_2$ with $h_2(0) = r >1$ with the condition that the corresponding disc $f$ is contained in $\Omega$. Using the defining function $\rho$, we 
need to ensure  
 $$-1+ \big|\zeta^8(1 + \zeta\varphi(\zeta))^2 - \zeta^9 h_2^3(\zeta)\big|^2 = -1+ \big|\zeta\big|^{16} \big|(1 + \zeta\varphi(\zeta))^2 - \zeta h_2(\zeta)^3\big|^2 <0,$$
that is,
 $$ \big|\zeta\big|^{16} \big|1 + 2\zeta\varphi(\zeta) + \zeta^2\varphi(\zeta)^2 - \zeta h_2(\zeta)^3\big|^2 < 1$$
This leads us to find  holomorphic functions $\varphi$ and $h_2$ satisfying the equation
\begin{equation}\label{eqkey}
\varphi(\zeta) (2+ \zeta\varphi(\zeta)) =  h_2(\zeta)^3
\end{equation} 
We then construct $\varphi$ so that both $\varphi(\zeta)$ and $2+ \zeta\varphi(\zeta)$ have no zero in the unit disc $\D$. Consider 
\begin{equation}\label{eqphi}
\varphi(\zeta)=\frac{e^\zeta-1}{\zeta}.
\end{equation}
This holomorphic function has no zeros on $\D$ and the same holds for   
$$2+ \zeta\varphi(\zeta)=1+e^\zeta.$$
We may then take for $h_2$ a cubic root of $\varphi(\zeta) (2+ \zeta\varphi(\zeta))$. It follows that the holomorphic disc  
$$\left(\zeta^4\left(1+\zeta \frac{e^\zeta -1}{\zeta}\right),\zeta^3h_2(\zeta),-1\right)=\left(\zeta^4e^\zeta,\zeta^3h_2(\zeta),-1\right)$$
is contained in $\Omega$, satisfies the conditions (\ref{eqcond}), and since $h_2(0)=2^{1/3}$  we obtain the estimate
$$K^3_\Omega\left((0,0,-1), (0,1,0)\right)\leq  \frac{1}{2^{1/3}}$$
which refines our basic estimate (\ref{eqbasic}). Following this approach, we now modify the above function $\varphi$ in (\ref{eqphi}) to sharpen this estimate.  Thus we consider 
$$\varphi(z) = \alpha \cfrac{e^{\beta \zeta} -1}{\zeta},$$
with $\alpha,\beta>0$.  
We have $\varphi(0) = \alpha \beta$, and $\varphi(\zeta) = 0$ if and only if  $\beta\zeta \in 2\pi i\Z\setminus\{0\}$. Therefore, to ensure that $\varphi$ has no zeros in $\D$, we need to have 
\begin{equation}\label{eqbeta}
\beta \leq 2 \pi.
\end{equation}
We turn to the function 
 $$2+ \zeta\varphi(\zeta)=2-\alpha+\alpha e^{\beta \zeta}$$
 which vanishes when $\displaystyle e^{\beta \zeta}=\frac{\alpha-2}{\alpha}$. To make sure that $2+ \zeta\varphi(\zeta)$ has no zeros in $\D$, we need 
\begin{equation}\label{eqalpha} 
\log\left(\cfrac{\alpha-2}{\alpha}\right) \leq -\beta.
\end{equation}
Moreover,  using Equation (\ref{eqkey}), we have  
$$r^3=h_2^3(0)=2\alpha \beta.$$
The value $r$ is the largest when equality occurs in (\ref{eqbeta}) and (\ref{eqalpha}), namely when $\beta=2\pi$ and $\alpha=\cfrac{2}{1 - e^{-2\pi}}$. This construction leads us to consider the holomorphic disc
 $$\displaystyle f(z) = \left(\zeta^4\left(1 +\frac{2(e^{2\pi \zeta} -1)}{1 - e^{-2\pi}}\right), \zeta^3h_2(\zeta), -1\right),$$
 where $h_2$ is a well defined holomorphic disc defined by Equation (\ref{eqkey}). The constructed disc is contained in $\Omega$, satisfies the conditions (\ref{eqcond}), and since 
$h_2(0)=\left(\cfrac{8\pi}{1-e^{-2\pi}}\right)^{1/3}$  we obtain the desired estimate
\begin{equation*}
K^3_\Omega\left((0,0,-1), (0,1,0)\right)\leq \left(\cfrac{8\pi}{1-e^{-2\pi}}\right)^{-1/3}.
\end{equation*}
\end{proof}

\begin{rem}
Note that instead of $f_1(\zeta) = \zeta^4h_1(\zeta)$, one could have considered a disc of the form $f_1(\zeta) = \zeta^nh_1(\zeta)$ with $n >4$. However, in case $n>4$, the condition to ensure that the disc 
$\left(\zeta^nh_1(\zeta),\zeta^3h_2(\zeta),-1\right)$ is contained in the domain $\Omega$ becomes  
$$\left| f_1^2 - f_2^3\right|^2  = |\zeta|^{18}\left|\zeta^{2n-9}h_1^2(\zeta) - h_2^3(\zeta)\right|^2 < 1.$$
A classical subharmonicity argument (see e.g. p.105 \cite{yu}) implies $|h_2(0)|=r \leq 1$, which does not improve the estimate we have obtained in Theorem \ref{theoest}.
\end{rem}

\subsection{On the standard Kobayashi metric}
As an interesting application of the method developed in the proof of Theorem \ref{theoest}, we are able to obtain the exact value of the standard Kobayashi metric in the domain $\Omega$ in certain new cases. 

Consider the point $z_t=(0,0,-t)$ with $0<t<1$, and the vector $X=(a,b,0)$ with $|a|^2+|b|^2=1$. In case $a \neq 0$, the following estimate was obtained by Yu in \cite{yu}:
\begin{equation}\label{eqeq}
|a|t^{-\frac{1}{4}} \leq K_{\Omega}(z_t,X)\leq t^{-\frac{1}{4}}.
\end{equation}
As pointed out by Jarnicki and Pflug in p.145 \cite{ja-pf}, it would be interesting to know the exact value of $K_{\Omega}(z_t,X)$. We are able to answer this question in some cases. 
\begin{prop}\label{propexact}
Assume that $\displaystyle \frac{|b|^3}{|a|^3}\leq \frac{2}{\sqrt{t}}\min\left\{2\pi,\log \left(1+2t^{\frac{1}{4}}\right)\right\}$. Then we have  
$$  K_{\Omega}(z_t,X)= |a|t^{-\frac{1}{4}}.$$
\end{prop} 
Essentially, this proposition states that for any point $z_t=(0,0,-t) \in \Omega$ with $0<t<1$ on the normal line through the origin, there exists a region of directions for which we know the exact value of the Kobayashi metric. 
\begin{proof}
We follow the strategy of the proof of Theorem \ref{theoest}.
Consider the holomorphic disc $f=(f_1,f_2,-t)$ of the form 
$$f(\zeta)=\left(\zeta \left(\frac{a}{|a|t^{-\frac{1}{4}}}+\zeta \varphi(\zeta)\right),\zeta h_2(\zeta) -t\right),$$
where
$$\displaystyle \varphi(\zeta)=\frac{e^{\frac{b^3\sqrt{t}}{2a|a|^2}\zeta}-1}{\zeta},$$
and where $h_2$ is such that
\begin{equation}\label{eqkey2}
\varphi(\zeta)\left(2\frac{a}{|a| t^{-\frac{1}{4}}}+\zeta \varphi(\zeta)\right) =  h_2^3(\zeta)
\end{equation} 
We first show that $h_2$ is a well defined holomorphic disc. We  note that $\varphi$ has no zero in the unit disc. We need to show that the same occurs for 
$$2\frac{a}{|a|t^{-\frac{1}{4}}}+\zeta \varphi=2\frac{a}{|a|t^{-\frac{1}{4}}}+e^{\frac{b^3\sqrt{t}}{2a|a|^2}\zeta}-1.$$ Note that on $\partial \D$, we have
$$\left|e^{\frac{b^3\sqrt{t}}{2a|a|^2}\zeta}-1\right|=\left|\sum_{k\geq1} \left(\frac{b^3\sqrt{t}}{2a|a|^2}\right)^k\frac{\zeta^k}{k!}\right |\leq \sum_{k\geq1} \frac{1}{k!}\left(\frac{|b|^3\sqrt{t}}{2|a|^3}\right)=e^{\frac{|b|^3\sqrt{t}}{2|a|^3}}-1.$$
The latter is less than $2t^{\frac{1}{4}}$ since  $\frac{|b|^3\sqrt{t}}{|a|^3}\leq \log \left(1+2t^{\frac{1}{4}}\right)$. By Rouch\'e's theorem, it follows that $2\frac{a}{|a|t^{-\frac{1}{4}}}+\zeta \varphi$ has no zeros in the unit disc. Thus the function $h_2$, and so the disc $f$, are well defined. 

Moreover, by construction we have $f(\D) \subset \Omega$. Indeed, using (\ref{eqkey2}) we have
\begin{eqnarray*}
 -t+ \big|f_1^2(\zeta)-f_2^3(\zeta)\big|^2  &=&  -t+ |\zeta|^4\left|\left(\frac{a}{|a|t^{-\frac{1}{4}}}+\zeta \varphi(\zeta)\right)^2-\zeta h_2^3(\zeta)\right|^2\\
\\
  &=&  -t+ |\zeta|^4\left|\left(\frac{a}{|a|t^{-\frac{1}{4}}}\right)^2\right|^2= -t+ |\zeta|^4t<0.\\
 \end{eqnarray*}
Finally, we have $f(0)=z_t$ and $f'(0)=\frac{1}{ |a|t^{-\frac{1}{4}}}X$. Combined with (\ref{eqeq}), this  proves the proposition. 
  
\end{proof}         

\section{On higher order extremal discs}

In this section, we are interested in the structure of the set of $k$-extremal discs.   For $a\in \D$, we denote by $B_a:\D\to \D$ the Blaschke function $B_a(\zeta)=\cfrac{\zeta-a}{1-\overline{a}\zeta}$.

\subsection{General results}

We first start with the following proposition. Recall that for a domain $\Omega \subset \C^n$, the set of $k$-extremal discs for the pair $(p,v)$ is denoted by $X_{\Omega}^k(p,v)$. 
\begin{lem}\label{lemext} 
Let $\Omega \subset \C^n$ be a bounded domain. Let $k>0$ be a positive integer and  $p\in \Omega$ and $v \in T_p\Omega$. 
Then $K^k_{\Omega}(p,v) = K_{\Omega}(p,v)$ if and only if $\{f(\zeta^k) \ | \ f \in X_{\Omega}(p,v)\} \subset X_{\Omega}^k(p,v)$.
\end{lem}

\begin{proof}
Let $p \in \Omega$ and $v \in T_p\Omega$.
Assume first that $K^k_{\Omega}(p,v) = K_{\Omega}(p,v)$. Let $f\in X_{\Omega}(p,v)$ and set $h(\zeta)=f(\zeta^k)$.  We note that $h(0)=p$, $\nu(h)\geq k$, 
$$h^{(k)}(0)=k!f^{(k)}(0)=\frac{k!}{K_{\Omega}(p,v)}v=\frac{k!}{K^k_{\Omega}(p,v)}v$$
which shows directly that $h \in X_{\Omega}^k(p,v)$.
     
Now assume that $\{f(\zeta^k) \ | \ f \in X_{\Omega}(p,v)\} \subset X_{\Omega}^k(p,v)$. Let $f$ be an extremal disc for the pair $(p,v)$ and define the disc $h(\zeta)=f(\zeta^k)$.  As before we have 
$h(0)=p$, $\nu(h)\geq k$, and 
$$h^{(k)}(0)=\frac{k!}{K_{\Omega}(p,v)}v,$$
and since $h\in X_{\Omega}^k(p,v)$, we have  $h^{(k)}(0)=\frac{k!}{K^k_{\Omega}(p,v)}v$, leading to  $K^k_{\Omega}(p,v) = K_{\Omega}(p,v)$. 
\end{proof}
In general, the equality does not hold, even if $K^k_{\Omega}(p,v) = K_{\Omega}(p,v)$. 
\begin{ex}
Consider the bidisc $\D\times \D \subset \C^2$. Note that 
$$K^2_{\D\times \D}((0,0),(1,0))= K_{\D\times \D}((0,0),(1,0))$$
by the product property of the higher order Kobayashi pseudometric (see e.g. Proposition 3.8.7 in \cite{ja-pf}) and by the fact that $K^2_{\D} \equiv K_{\D}$. 
According to Schwarz Lemma, we have  
$$X_{\D\times \D}((0,0),(1,0)) = \{ (\zeta, \zeta^2\psi(\zeta)) \ | \  \psi: \D \to \D \text{ holomorphic}\}.$$
 The holomorphic disc $\zeta \mapsto (\zeta^2 , \zeta^3) \in X_{\Omega}^2((0,0),(1,0))$ and is not of the form $f(\zeta^2)$ for some extremal disc $f \in X_{\Omega}((0,0),(1,0)).$
\end{ex}

It is interesting to note that in case $\Omega$ is the unit disc $\D$ or the punctured disc $\D\setminus\{0\}$, we have 
\begin{prop}
Let $\Omega=\D$ or $\D\setminus\{0\}$. Then for any $(p,v) \in \Omega \times T_p\Omega$, we have
$$\{f(\zeta^k) \ | \ f \in X_{\Omega}(p,v)\} = X_{\Omega}^k(p,v).$$ 
\end{prop}

\begin{proof}
According to Proposition \ref{propprop} and Lemma \ref{lemext} we only need to show that $X_{\Omega}^k(p,v) \subset \{f(\zeta^k) \ | \ f \in X_{\Omega}(p,v)\}$. This follows from the equality cases established in the corresponding higher order Schwarz lemmas (Lemma \ref{lemhigh} and Lemma \ref{lempdhigh}) and from the invariance of extremal discs (Lemma \ref{leminv}).
\end{proof}

\subsection{The case of complex ellipsoids in $\C ^2$}
The Kobayashi metric on convex and nonconvex complex ellipsoids has been studied by many authors \cite{po,bl-fa-kl-kr-ma-pa,ja-pf-ze,pf-zw,ed}. It is interesting to note that in the papers \cite{po,ja-pf-ze,pf-zw,ed}, the metric is studied by means of extremal discs. 

In this section we focus on the {\it nonconvex complex ellipsoid}
$$\mathcal{E}(1,m) = \{ (z_1,z_2) \in \C^2 \; \big| \;  |z_1|^{2} + |z_2|^{2m} < 1\}$$
where $m \in (0,1/2)$.
 We first note that $\mathcal{E}(1,m)$ is a bounded balanced pseudoconvex domain, that is if $z\in\mathcal{E}(1,m)$ and $\lambda\in\bar{\D}$ then $\lambda z\in\mathcal{E}(1,m)$. Accordingly, the higher order Kobayashi metric coincide with the Kobayashi metric at the origin $(0,0)$ (see e.g. Theorem 2.4 in \cite{ki}). 
 
\begin{theo}\label{theopext}
Let $(a,b) \in \mathcal{E}(1,m)$ and $v \in \C^2$. Let $f \in X_{\mathcal{E}(1,m)}\left((a,b),v\right)$. Then the map $f(\zeta^k)$ is $k$-stationary in the $L^{\infty}$ sense.    
\end{theo}

\begin{proof} 
Following \cite{pf-zw} and \cite{ed}, we consider two kind of maps (see Proposition 2 \cite{pf-zw} and Theorem 2 \cite{ed}). 
The first sort $\varphi=(\varphi_1,\varphi_2): \D \to \mathcal{E}(1,m)$ of is of the form 
\begin{equation}\label{eqkind1}
\varphi(\zeta)=\left(a_1B_{\alpha_1}^{r_1}(\zeta)\cdot  \frac{1-\overline{\alpha_1}\zeta}{1-\overline{\alpha_0}\zeta},
a_2 B_{\alpha_2}^{r_2}(\zeta)\cdot \left(\frac{1-\overline{\alpha_2}\zeta}{1-\overline{\alpha_0}\zeta}\right)^{\frac{1}{2m}}\right)
\end{equation}
with 
$$
\begin{cases}
a_1,a_2 \in \C^{*},  \alpha_0 \in \D,  \alpha_1,\alpha_2 \in \overline{\D}\\
\\
r_j \in \{0,1\}, j=1,2 \mbox{ and if} \ r_j=1 \mbox{ then}\ \alpha_j \in \D\\
\\
\alpha_0=|a_1|^2\alpha_1+|a_2|^{2m}\alpha_2\\
\\
1+|\alpha_0|^2=|a_1|^2(1+|\alpha_1|^2)+|a_2|^{2m}(1+|\alpha_2|^2)
\end{cases}
$$
The second kind of maps $\psi=(\psi_1,\psi_2): \D \to \mathcal{E}(1,m)$ is of the form 
\begin{equation}\label{eqkind2}
\psi(\zeta)=\left(a_1\prod_{\ell=1}^kB_{\alpha_{\ell 1}}^{r_{\ell 1}}(\zeta)\cdot \frac{1-\overline{\alpha_{\ell 1}}\zeta}{1-\overline{\alpha_{\ell 0}}\zeta},
a_2 \prod_{\ell=1}^k B_{\alpha_{\ell2}}^{r_{\ell2}}(\zeta)\cdot \left(\frac{1-\overline{\alpha_{\ell2}}\zeta}{1-\overline{\alpha_{0\ell}}\zeta}\right)^{\frac{1}{2m}}\right)
\end{equation}
with 
$$
\begin{cases}
a_1,a_2 \in \C\setminus\{0\},  \alpha_{\ell j} \in \overline{\D}, \ell=1,\ldots,k, j=0,1,2\\
\\
r_{\ell j} \in \{0,1\}, \ell=1,\ldots,k, j=1,2,  \mbox{ and if} \ r_{\ell j}=1 \mbox{ then} \ \alpha_{\ell j} \in \D\\
\\
\displaystyle |a_1|^2\prod_{\ell=1}^k(\zeta-\alpha_{\ell 1})(1-\overline{\alpha_{\ell 1}}\zeta)+|a_2|^{2m}\prod_{\ell=1}^k(\zeta-\alpha_{\ell2})(1-\overline{\alpha_{\ell2}}\zeta)
=\prod_{\ell=1}^k(\zeta-\alpha_{\ell 0})(1-\overline{\alpha_{\ell 0}}\zeta)
\\
\end{cases}
$$

Consider now an extremal disc $\varphi=(\varphi_1,\varphi_2) \in X_{\mathcal{E}(1,m)}\left((a,b),v\right)$. 
According to Pflug and Zwonek \cite{pf-zw}, the map $\varphi$  is necessarily of the form (\ref{eqkind1}). We wish to show that $\varphi(\zeta^k)$ is of the kind (\ref{eqkind2}).
For this purpose, for $j=0,1,2$, we consider  the $k^{th}$-roots of $\alpha_j$ and denote them by $\alpha_{1j},\alpha_{2j}\ldots,\alpha_{kj}$. We also set $r_{\ell j}=r_j$ for all  $\ell=1,\ldots,k$ and $ j=1,2$. 
Define 
$$I=\displaystyle |a_1|^2\prod_{\ell=1}^k(\zeta-\alpha_{\ell 1})(1-\overline{\alpha_{\ell 1}}\zeta)+|a_2|^{2m}\prod_{\ell=1}^k(\zeta-\alpha_{\ell2})(1-\overline{\alpha_{\ell2}}\zeta)
$$ 
We compute, using $\zeta^k-\alpha_j=\prod_{\ell=1}^k(\zeta-\alpha_{kj})$ and $1-\overline{\alpha_j}\zeta^k=\prod_{\ell=1}^k(1-\overline{\alpha_{kj}}\zeta^2)$, 
\begin{eqnarray*}
I&=&\displaystyle |a_1|^2(\zeta^k-\alpha_1)(1-\overline{\alpha_{1}}\zeta^k)+|a_2|^{2m}(\zeta^k-\alpha_2)(1-\overline{\alpha_{2}}\zeta^k)\\
\\
& =& -\left(|a_1|^2\overline{\alpha_{1}}+|a_2|^{2m}\overline{\alpha_{2}}\right)\zeta^{2k}+ \left(|a_1|^2(1+|\alpha_1|^2)+|a_2|^{2m}(1+|\alpha_2|^2)\right)\zeta^k  -\left(|a_1|^2\alpha_{1}+|a_2|^{2m}\alpha_{2}\right)\\
\\
& =& -\overline{\alpha_0}\zeta^{2k}+ \left(1+|\alpha_0|^2\right)\zeta^k  -\alpha_0=(\zeta^k-\alpha_0)(1-\overline{\alpha_0}\zeta^k)=\prod_{\ell=1}^k(\zeta-\alpha_{\ell 0})(1-\overline{\alpha_{\ell 0}}\zeta).\\
\end{eqnarray*}
and
\begin{eqnarray*}
\varphi(\zeta^k)&=&\left(a_1B_{\alpha_1}^{r_1}(\zeta^k)\cdot  \frac{1-\overline{\alpha_1}\zeta^k}{1-\overline{\alpha_0}\zeta^k},
a_2 B_{\alpha_2}^{r_2}(\zeta^k)\cdot \left(\frac{1-\overline{\alpha_2}\zeta^k}{1-\overline{\alpha_0}\zeta^k}\right)^{\frac{1}{2m}}\right)\\
\\
& =&\left(a_1\prod_{\ell=1}^kB_{\alpha_{\ell 1}}^{r_{\ell 1}}(\zeta)\cdot \frac{1-\overline{\alpha_{\ell 1}}\zeta}{1-\overline{\alpha_{\ell 0}}\zeta},
a_2 \prod_{\ell=1}^k B_{\alpha_{\ell2}}^{r_{\ell2}}(\zeta)\cdot \left(\frac{1-\overline{\alpha_{\ell2}}\zeta}{1-\overline{\alpha_{0\ell}}\zeta}\right)^{\frac{1}{2m}}\right).\\
\end{eqnarray*}
Therefore the map $\varphi(\zeta^k)$ is exactly of the form (\ref{eqkind2}). Following a variational approach due to Poletsky \cite{po},  
Edigarian \cite{ed} showed that maps of the form (\ref{eqkind2}) are solutions of a certain Euler-Lagrange equation 
(see Problem  ($\mathcal{P}$) of $m$-type p.84 \cite{ed}). Together with Remark 11.4.4 in \cite{ja-pf}, it follows that such maps $k$-stationary in the  $L^{\infty}$ (see also \cite{be-de-jo}). 
This shows that $\varphi(\zeta^k)$ is then $k$-stationary in the  $L^{\infty}$.  
\end{proof}

\begin{rem}
We could have focused on maps of the form (\ref{eqkind1}) and (\ref{eqkind2}) centered at $(0,b)$. Indeed, we know (see e.g. \cite{pf-zw}) that for $a \in \D$ and $\theta \in \R$, the  map $F_{a,\theta}$ defined by 
$$F_{a,\theta}(z_1,z_2)= \left(\cfrac{z_1 - a}{1 - a\bar z_1} , \cfrac{e^{i \theta}(1 - |a|^2)^{\frac{1}{2m}} z_2}{(1 - a\bar z_1)^{\frac{1}{m}}}\right)$$
 is an automorphism of $\mathcal{E}(1,m)$ which maps any point $(a,c) \in \mathcal{E}(1,m)$ to a point $(0,b) \in \mathcal{E}(1,m)$ with $b \in [0,1)$. Moreover, by Lemma \ref{leminv}, we have $$ (F_{a, \theta})_* X^k_{\mathcal{E}(1,m)}\left((a,c), d_{(0,0)}F_{a, \theta}(v)\right) =X^k_{\mathcal{E}(1,m)} \big((0,b),  v\big).$$
Nevertheless, we decided to keep the more general form of such maps since the simplification of notation is barely noticeable. 
\end{rem}

\begin{rem}
In case the complex ellipsoid $\mathcal{E}(1,m)$ is convex, that is, when $m>\frac{1}{2}$. We know from Lempert theory \cite{le} that the Kobayashi metric and the Carath\'eodory metric 
coincide. Therefore by Proposition \ref{propprop}, the higher order Kobayashi metrics and the standard Kobayashi metric coincide; this fact was already observed by many authors 
\cite{ja-pf,yu}. According to Lemma \ref{lemext}, we have
$$\left\{f(\zeta^k) \ | \ f \in X_{\mathcal{E}(1,m)}\left((a,b),v\right)\right\}  \subset  X^k_{\mathcal{E}(1,m)}\left((a,b),v\right).$$
It follows that if $\varphi$ is a map of the form (\ref{eqkind1}), then $\varphi(\zeta^k)$ is a $k$-extremal and, thus, must be of the form (\ref{eqkind2}). Nevertheless, without a direct computation, the dependance of the parameters $\alpha_{\ell j}$ on the  parameters $\alpha_j$ is only implicit.   
\end{rem}

\section{Appendix: on higher order Schwarz lemmas}
In this appendix, we establish higher order versions of Schwarz type lemmas in the vein of Lemma \ref{lemhigh}.

 \begin{lem}\label{lemsphigh}
Let $f: \D \rightarrow \D$ be a holomorphic function and $\zeta \in \D$ with $f^{(\ell)}(\zeta) = 0$ for all $ \ell=1,\ldots,k-1$. We have for all $w \in \D$ 
\begin{equation}\label{eqsphigh1}
\left|\cfrac{f(\zeta)-f(w)}{1 - \overline{f(\zeta)}f(w)}\right| \leq  \bigg|\cfrac{\zeta-w}{1-\Bar{\zeta}w}\bigg|^k
\end{equation}
and 
\begin{equation}\label{eqsphigh2}
\left|f^{(k)}(\zeta)\right| \leq k! \cfrac{1-|f(\zeta)|^2}{(1-|\zeta|^2)^k}.
\end{equation}
\noindent 
Moreover, in case of equality (with $w\neq \zeta$) then $f$ is of the form 
$f(w) = \cfrac{e^{i\theta}(B_{\zeta}(w))^k + f(\zeta)}{1 + e^{i\theta}\overline{f(\zeta)}(B_{\zeta}(w))^k}.$ 
\end{lem}
\begin{proof}
Let $\zeta,w \in \D$. Set $a=f(\zeta) $ and consider the function $g= B_a\circ f\circ B_{-\zeta}:\D \to \D$. We have  $g(0) = 0$ and, for  $\ell=1,\ldots,k-1$ 
$$g^{(\ell)}(0) = (B_a \circ f\circ B_{-\zeta})^{(\ell)}(0) = 0$$
since $(f\circ B_{-\zeta})^{(\ell)}(0) = 0$. Moreover, a straightforward computation also leads to  
$$g^{(k)}(0) =   B_a'((f\circ B_{-\zeta})(0))\cdot (f\circ B_{-\zeta})^{(k)}(0) = B_a'(a)\cdot f^{(k)}(\zeta) \cdot (B_{-\zeta}'(0))^k.$$
By applying Lemma \ref{lemhigh} with $g$, we obtain  for $\tilde{w}\in \D$ 
$$|g(\tilde{w})| = \bigg|\cfrac{f(B_{-\zeta}(\tilde{w})) - a}{1 - \overline{a}f(B_{-\zeta_1}(\tilde{w}))}\bigg| \leq |\tilde{w}|^k,$$
which, for  $\tilde{w}= B_{\zeta}(w)$, gives (\ref{eqsphigh1}), and 
 $$\left|g^{(k)}(0)\right| = \left|f^{(k)}(\zeta)\right| \cdot \left|B_a'(a)\right| \cdot \left|(B_{-\zeta}'(0))\right|^k \leq k!$$
 which implies (\ref{eqsphigh2}). 

 Finally, if equality holds, then by Lemma \ref{lemhigh} we have $g(\zeta) = e^{i\theta}\zeta^{k}$ for some $\theta \in \R$, and thus 
 $$f(w) = B_{-a}\circ g \circ B_{\zeta}(w) = \cfrac{e^{i\theta}(B_{\zeta}(w))^k + f(\zeta)}{1 + e^{i\theta}\overline{f(\zeta)}(B_{\zeta}(w))^k}.$$ \end{proof}

We  also establish a higher order Schwarz Lemma in the case of  the punctured disc.

\begin{lem}\label{lempdhigh}
Let $f: \D \rightarrow \D\setminus\{0\}$ be a holomorphic function such that $f^{(\ell)}(0) = 0$ for all $\ell=1,\ldots,k-1$. Then
\begin{equation}\label{eqpdhigh} 
 \left|f^{(k)}(0)\right| \leq -2k!|f(0)|\log|f(0)|.
 \end{equation}
 Moreover in case of equality then $f$ is of the form   
 \begin{equation*}
 f(\zeta) = e^{i\alpha}  e^{ \log|f(0)|\cfrac{1+e^{i\theta}\zeta^k}{1 - e^{i\theta}\zeta^k}}
 \end{equation*}
 for some $\theta \in \R$ and where $f(0)=|f(0)|e^{i\alpha}$. 
\end{lem}

\begin{proof} We write $f$ as an exponential $f=e^g$ for some holomorphic function 
$$g: \D \to \left\{ \zeta \in \C \ | \ \Re e\zeta<0\right\}.$$ 
It follows that $\varphi \circ (-ig)$, where $\varphi:\HH \to \D$ is the Cayley transform  $\varphi(\zeta)=\frac{\zeta-i}{\zeta+i}$,   is a self map of the unit disc. For $\ell=1,\ldots,k-1$ we have $g^{(\ell)}(0)$ and thus, after a direct computation we obtain 
$$(\varphi \circ (-ig))^{(\ell)}(0) = 0$$
and 
$$(\varphi \circ (-ig))^{(k)}(0) = \varphi'(-ig(0))  g^{(k)}(0) = \frac{\varphi'(-ig(0))   f^{(k)}(0)}{f(0)}.$$
We now apply Lemma \ref{lemsphigh} to $\varphi \circ (-ig)$ and get 

$$\bigg|\cfrac{(\varphi \circ (-ig))^{(k)}(0)}{k!}\bigg| = \frac{1}{k!} \bigg|\cfrac{2}{(g(0) - 1)^2}\, \cdot \,  \cfrac{f^{(k)}(0)}{ f(0)}\bigg| \leq 1 - |(\varphi\circ (-ig))(0)|^2 = 1 - \left| \cfrac{g(0) +1 }{g(0) -1}\right|^2.$$
This implies
$$\frac{2}{k!}\left|\cfrac{f^{(k)}(0)}{ f(0)}\right| \leq |1 - g(0)|^2 - |1 + g(0)|^2 = 2(-g(0) - \overline{g(0)}) = -4\Re e g(0) = -4\log|f(0)|$$
and proves (\ref{eqpdhigh}). 

We now write  $f(0) = |f(0)|e^{i\alpha}$ with $0 \leq \alpha \leq 2\pi$. Then,  we can renormalize the holomorphic function g such that $ f = e^{i\alpha + g}$ where $g(0) = \log|f(0)| \in \R$. 
In case of equality in (\ref{eqpdhigh}), then by Lemma \ref{lemsphigh} we have  
$$\varphi \circ (-ig) = B_{-\varphi(-ig(0))} \circ  (e^{i\theta}\zeta^k)$$ 
and so
$$ g(\zeta) =- \cfrac{B_{-\varphi(-ig(0))}\circ (e^{i\theta}\zeta^k) + 1}{1 - B_{-\varphi(-ig(0))} \circ (e^{i\theta}\zeta^k)} = 
g(0)\cfrac{e^{1+i\theta}\zeta^k}{1 - e^{i\theta}\zeta^k}$$
leading to (\ref{eqpdhigh}). 
\end{proof}

\noindent  {\it Acknowledgments.}  This work was done in the framework of the  Summer Research Camp in Mathematics designed by the Department of Mathematics at the American University of Beirut (AUB) and that benefitted from a generous support from the Center for Advanced Mathematical Sciences at AUB.

{\small
\noindent Seok Ban, Florian Bertrand, Amir Jaber Chehayeb, Adam Salha, Walid Tabbara\\
Department of Mathematics\\\
American University of Beirut, Beirut, Lebanon\\{\sl E-mail addresses}: sjb11@mail.aub.edu, fb31@aub.edu.lb, amj24@mail.aub.edu, ars23@mail.aub.edu, wkt00@mail.aub.edu\\

}

\end{document}